\documentclass[12pt]{amsproc}
\usepackage{geometry, hyperref, verbatim, enumerate,amsthm, amsmath, amssymb, amscd, mathrsfs, color, tikz-cd}

\title{Integer-valued polynomials on subsets of quaternion algebras}

\author{Nicholas J. Werner}
\address{Department of Mathematics, Computer and Information Science, SUNY at Old Westbury, Old Westbury, NY 11568,USA}
\email{wernern@oldwestbury.edu}

\numberwithin{equation}{section}

\theoremstyle{definition}\newtheorem{Def}[equation]{Definition}
\theoremstyle{plain}\newtheorem{Lem}[equation]{Lemma}
\theoremstyle{plain}\newtheorem{Prop}[equation]{Proposition}
\theoremstyle{plain}\newtheorem{Thm}[equation]{Theorem}
\theoremstyle{plain}\newtheorem{Cor}[equation]{Corollary}
\theoremstyle{plain}
\theoremstyle{remark}\newtheorem{Rem}[equation]{Remark}
\theoremstyle{definition}\newtheorem{Ex}[equation]{Example}
\theoremstyle{definition}
\theoremstyle{definition}
\theoremstyle{definition}\newtheorem{Alg}[equation]{Algorithm}

\newcommand{\bfh}{\mathbf{h}}
\newcommand{\bfi}{\mathbf{i}}
\newcommand{\bfj}{\mathbf{j}}
\newcommand{\bfk}{\mathbf{k}}

\newcommand{\Int}{{\rm Int}}
\newcommand{\smat}[4]{[\begin{smallmatrix} #1 & #2 \\ #3 & #4 \end{smallmatrix}]}
\newcommand{\mcC}{\mathcal{C}}
\newcommand{\mcN}{\mathcal{N}}

\newcommand{\D}{\mathbb{D}}
\newcommand{\F}{\mathbb{F}}
\newcommand{\HQ}{\mathbf{H}}
\newcommand{\LQ}{\mathbf{L}}

\newcommand{\Q}{\mathbb{Q}}

\newcommand{\Z}{\mathbb{Z}}

\begin{document}

\maketitle

\begin{center}\today\end{center}

\begin{abstract}
Let $R$ be either the ring of Lipschitz quaternions, or the ring of Hurwitz quaternions. Then, $R$ is a subring of the division ring $\D$ of rational quaternions. For $S \subseteq R$, we study the collection $\Int(S,R) = \{f \in \D[x] \mid f(S) \subseteq R\}$ of polynomials that are integer-valued on $S$. The set $\Int(S,R)$ is always a left $R$-submodule of $\D[x]$, but need not be a subring of $\D[x]$. We say that $S$ is a ringset of $R$ if $\Int(S,R)$ is a subring of $\D[x]$. In this paper, we give a complete classification of the finite subsets of $R$ that are ringsets.
\end{abstract}

\section{Introduction}
\thispagestyle{empty}

In this paper, we study integer-valued polynomials on subsets of noncommutative rings. In the commutative setting, these rings emerge by letting $D$ be a (commutative) integral domain with fraction field $K$, taking $S \subseteq D$, and defining
\begin{equation*}
\Int(S,D) := \{f \in K[x] \mid f(S) \subseteq D\},
\end{equation*}
which is the ring of integer-valued polynomials sending $S$ into $D$. When $S=D$, we let $\Int(D):=\Int(D,D)$. Work on the ring-theoretic aspects of $\Int(S,D)$ dates to at least the 1970s. The book \cite{CahenChabert} is the standard reference for results in this field up to the late 1990s.

During the past fifteen years, attention has turned to similar constructions over noncommutative rings. Integer-valued polynomials have been studied over various quaternion algebras \cite{CigliolaLoper, CigliolaSpirito, JohnsonPav, WernerQuaternion, WernerSurvey}, matrix algebras \cite{Frisch2013, NagRisHaf2017, Peruginelli, HafNagRis2020, HafNagSak2019, HafNag2021,  WernerMatrix}, and rings of upper triangular matrices \cite{EvrardFaresJohnson, Frisch2017, NagHaf2024}. When $D$ is a commutative domain, it is clear that $\Int(D)$ and $\Int(S,D)$ are subrings of $K[x]$. However, this may not be the case when working with polynomials with noncommuting coefficients; see Example \ref{ex:not a ringset} below.

Let $B$ be a noncommutative ring. We will follow standard conventions for working with the noncommutative polynomial ring $B[x]$ as in \cite[\S 16]{Lam}. 
In $B[x]$, polynomials are added and multiplied as in the commutative case, and the indeterminate $x$ is central in $B[x]$. The main difference from the commutative setting is that we will assume polynomials satisfy right evaluation, i.e.\ that before a polynomial $f \in B[x]$ can be evaluated, it must be written as $f(x) = \sum_i a_i x^i$, where the indeterminate appears to the right of any coefficients. Because of this, evaluation at an element of $B$ may not define a multiplicative map $B[x] \to B$. For instance, let $a, b \in B$, $f(x)=ax$, and $g(x)=bx$. Denote the product of $f$ and $g$ in $B[x]$ by $(fg)(x)$. Then, $(fg)(x)=abx^2$, so $(fg)(\alpha) = ab\alpha^2$ for $\alpha \in B$, and this may fail to equal $f(\alpha)g(\alpha)=(a\alpha)(b\alpha)$. Note that if $\alpha \in B$ is central, then it is true that $(fg)(\alpha)=f(\alpha)g(\alpha)$ for all $f, g \in B[x]$. Also, if each coefficient of $g$ is central in $B$, then $(gf)(\alpha) = (fg)(\alpha)=f(\alpha)g(\alpha)$ for all $f \in B[x]$ and all $\alpha \in B$.

When $A$ is a subring of $B$ and $S \subseteq A$, we can define the sets of polynomials
\begin{equation*}
\Int(S,A) := \{f \in B[x] \mid f(S) \subseteq A\}
\end{equation*}
and $\Int(A):=\Int(A,A)$ as in the commutative case. It is straightforward to check that $\Int(S,A)$ is a left $A$-module, but $\Int(S,A)$ may not be closed under multiplication, and hence may fail to be a ring. One of the basic problems in the study of noncommutative integer-valued polynomials is to determine when $\Int(S,A)$ is a ring. For this, we focus on the case where $A$ is an associative, torsion-free $D$-algebra such that $A \cap K = D$, and we take $B = K \otimes_D A$, which is the extension of $A$ to a $K$-algebra.

\begin{Def}\label{def:ringset}
A subset $S \subseteq A$ is called a \textit{ringset} of $A$ if $\Int(S,A)$ is a ring.
\end{Def}

\begin{Ex}\label{ex:not a ringset}\cite[Ex.\ 39]{WernerSurvey} 
Assume that $D$ is a commutative Noetherian domain, and that $A$ is a noncommutative $D$-algebra. Choose $\alpha, \beta \in A$ such that $\alpha\beta \ne \beta\alpha$. We show that $S=\{\alpha\}$ is not a ringset. Since $D$ is Noetherian, there exists $d \in D\setminus\{0\}$ such that $\alpha\beta-\beta\alpha\notin dA$. Take $f(x)=(x-\alpha)/d$ and $g(x)=x-\beta$. Both polynomials are in $\Int(S,A)$, but $(fg)(x) = (x^2 - (\alpha+\beta)x + \alpha\beta)/d$, and $(fg)(\alpha) = (\alpha\beta-\beta\alpha)/d \notin A$. Thus, $fg \notin \Int(S,A)$, $\Int(S,A)$ is not a ring, and $S$ is not a ringset.
\end{Ex}

In contrast to the situation with singleton sets, there are many cases for which $A$ itself is a ringset.

\begin{Thm}\label{thm:sums of units}\cite[Thm.\ 1.2]{WernerMatrix}
Let $A$ and $B$ be as above. Assume that each $\alpha \in A$ can be written as a finite sum $\alpha = \sum_i c_i u_i$, where $c_i, u_i \in A$ are such that each $c_i$ is central in $B$ and each $u_i$ is a unit of $A$. Then, $\Int(A)$ is a ring.
\end{Thm}

Theorem \ref{thm:sums of units} applies to many types of $D$-algebras, including matrix rings and group algebras over $D$. Even when $A$ fails to satisfy the hypothesis of the theorem, $A$ may still be a ringset. For instance, if $A$ is a ring of upper triangular matrices over $D$, then there may exist elements of $A$ that cannot be written in the form $\sum_i c_i u_i$. Nevertheless, it is known \cite[Thm.\ 5.4]{Frisch2017} that $\Int(A)$ is a ring. To date, no example has been given of a $D$-algebra $A$ for which $\Int(A)$ is not a ring.

The purpose of this article is to study finite ringsets of two quaternion algebras over $\Z$. Let $\bfi$, $\bfj$, and $\bfk$ be the standard quaternion units that satisfy $\bfi^2=\bfj^2=-1$ and $\bfi\bfj = \bfk = -\bfj\bfi$. The ring of Lipschitz quaternions $\LQ$ and the ring of Hurwitz quaternions $\HQ$ are
\begin{align*}
\LQ &:= \{a_0+a_1\bfi+a_2\bfj+a_3\bfk \mid a_i \in \Z \text{ for each } i\}, \text{ and}\\
\HQ &:= \{a_0+a_1\bfi+a_2\bfj+a_3\bfk \mid a_i \in \Z \text{ for each } i, \text{ or } a_i \in \Z+\tfrac{1}{2} \text{ for each } i\}.
\end{align*}
Clearly, $\LQ \subseteq \HQ$, and both $\LQ$ and $\HQ$ are subrings of the division ring
\begin{equation*}
\D:= \{a_0+a_1\bfi+a_2\bfj+a_3\bfk \mid a_i \in \Q \text{ for each } i\}.
\end{equation*}
Basic properties of the arithmetic of $\LQ$, $\HQ$, and $\D$ can be found in \cite[Chap.\ 11]{Voight}. Note that both $\LQ$ and $\HQ$ have center $\Z$. The elements in $\HQ \setminus \LQ$ are exactly those of the form $(a_0+a_1\bfi+a_2\bfj+a_3\bfk)/2$ with $a_i$ an odd integer for each $0 \leq i \leq 3$. The unit group $\LQ^\times$ of $\LQ$ is $\LQ^\times = \{\pm 1, \, \pm \bfi, \pm \bfj, \pm \bfk\}$, which is the well-known quaternion group $Q_8$. The ring $\HQ$ contains the additional Hurwitz unit $\bfh:=(1+\bfi+\bfj+\bfk)/2$, and the full unit group $\HQ^\times$ has order 24 and is generated by $\bfi$ and $\bfh$.

For the remainder of this paper, $R$ will be either $\LQ$ or $\HQ$. The majority of our results and proofs hold for both $\LQ$ and $\HQ$, so we will distinguish between the two rings only when necessary. By \cite[Thm.\ 2.3]{WernerQuaternion}, $\Int(R)$ is a ring, so $R$ is a ringset of itself. Our goal is to characterize the finite subsets of $R$ that are ringsets. That is, we will describe all finite $S \subseteq R$ such that
\begin{equation*}
\Int(S,R) := \{f \in \D[x] \mid f(S) \subseteq R\}
\end{equation*}
is a subring of $\D[x]$. Similar sets of polynomials have been studied previously in the context of matrix algebras \cite{HafNagSak2019} and upper triangular matrix algebras \cite{NagHaf2024}. Some basic results on ringsets of $R$ were proved in \cite{WernerSurvey}.

\begin{Lem}\label{lem:basic ringsets}\cite[Sec.\ 6.3]{WernerSurvey}
Let $S, T \subseteq R$.
\begin{enumerate}[(1)]
\item A singleton set $S=\{\alpha\}$ is a ringset if and only if $\alpha \in \Z$.
\item For each $\alpha \in R$, the conjugacy class $\{u\alpha u^{-1} \mid u \in R^\times\}$ is a ringset.
\item If $S$ and $T$ are ringsets, then $S \cup T$ is a ringset.
\item If $S$ is a union of conjugacy classes, then $S$ is a ringset.
\end{enumerate}
\end{Lem}

\begin{Ex}\label{ex:i and j}\mbox{}
\begin{itemize}
\item When $R = \LQ$, the full conjugacy class of $\bfi$ in $R$ is $\{\pm \bfi\}$. So, $\{\pm \bfi\}$ is a ringset with respect to $\LQ$. When $R = \HQ$, the full conjugacy class of $\bfi$ is $\{\pm \bfi, \pm \bfj, \pm \bfk\}$, so this set is a ringset of $\HQ$. 

\item Let $S = \{\pm \bfi\}$. Then, $S$ is not a ringset with respect to $\HQ$. To see this, let $f(x) = (x-\bfi)/2 \in \Int(S,\HQ)$. Then, $(f\bfh)(x) = (\bfh x - \bfi \bfh)/2$, so $(f\bfh)(\bfi) = (\bfh \bfi - \bfi \bfh)/2 = (\bfj - \bfk)/2 \notin \HQ$. Thus, $f\bfh \notin \Int(S,\HQ)$ and $S$ is not a ringset with respect to $\HQ$.

\item It will follow from our later results (see Theorem \ref{thm:Gamma is 1, 2, or 4}) that every $S \subseteq \{\pm \bfi, \pm \bfj, \pm \bfk\}$ such that $|S| \geq 2$ is a ringset with respect to $\LQ$. So, both $\{\pm \bfi\}$ and $\{\bfi, \bfj\}$ are ringsets, but $\{\bfi\} = \{\pm \bfi\} \cap \{\bfi, \bfj\}$ is not a ringset by Lemma \ref{lem:basic ringsets}(1). This shows that intersections of ringsets need not be ringsets.
\end{itemize}
\end{Ex}

The first step in our classification of finite ringsets of $R$ is to reduce the problem to sets in which each element satisfies the same polynomial over $\Z$.

\begin{Def}\label{def:min polys}
Given $\alpha = a_0+a_1\bfi+a_2\bfj+a_3\bfk \in \D$, the \textit{real part} of $\alpha$ is $a_0$, and we call $a_1$, $a_2$, and $a_3$ the \textit{imaginary coefficients} of $\alpha$. Let $\overline{\alpha}:= a_0-a_1\bfi-a_2\bfj-a_3\bfk$. The norm of $\alpha$ is $N(\alpha):= \alpha \cdot \overline{\alpha} = a_0^2 + a_1^2 + a_2^2 + a_3^2$. The minimal polynomial $\mu_\alpha \in \Q[x]$ of $\alpha$ is
\begin{equation*}
\mu_\alpha(x) := \begin{cases} x-\alpha, & \alpha \in \Q\\ x^2 - 2a_0x + N(\alpha), & \alpha \in \D\setminus \Q. \end{cases}
\end{equation*}
When $\alpha \in R$, $\mu_\alpha$ will have integer coefficients. Given a monic polynomial $m \in \Z[x]$ of degree 1 or 2, we define $\mcC_R(m) := \{\alpha \in R \mid \mu_\alpha = m\}$, which is the \textit{minimal polynomial class} of $m$ in $R$.
\end{Def}

\begin{Rem}\label{C_R(m) is finite}
Observe that any minimal polynomial class $\mcC_R(m)$ in $R$ is a finite set. Indeed, for any $\alpha=a_0+a_1\bfi+a_2\bfj+a_3\bfk \in R$, we have $N(\alpha)=a_0^2+a_1^2+a_2^2+a_3^2$. Since each $a_i$ is either an integer or a half-integer, for any given norm $n$, there are only finitely many $\alpha \in R$ such that $N(\alpha)=n$. Thus, there are only finitely many elements of $R$ that satisfy $m$. Hence, $\mcC_R(m)$ must be finite.
\end{Rem}

Note that since $\D$ is a division ring, $\mu_\alpha$ is irreducible over $\Q$ for all $\alpha \in \D$. Also, it is clear that conjugate elements of $R$ share the same minimal polynomial, but the converse is not true. For example, $3\bfi$ and $\bfi+2\bfj+2\bfk$ both have minimal polynomial $x^2+9$, but they are not conjugate in $R$. Thus, the sets $\mcC_R(m)$ need not coincide with conjugacy classes in $R$. However, $\mcC_R(m)$ is always a union of conjugacy classes.

\begin{Lem}\label{lem:C(m) is a ringset}
Let $m \in \Z[x]$ be monic and have degree 1 or 2. Then,
\begin{equation*}
\mcC_R(m) = \bigcup_{\alpha \in \mcC_R(m)} \{u \alpha u^{-1} \mid u \in R^\times\},
\end{equation*}
and $\mcC_R(m)$ is a ringset.
\end{Lem}
\begin{proof}
It is clear that $\mcC_R(m)$ is contained in the above union of conjugacy classes. Conversely, if $\beta \in \{u \alpha u^{-1} \mid u \in R^\times\}$ for some $\alpha \in \mcC_R(m)$, then $\beta \in \mcC_R(m)$. Thus, the given decomposition of $\mcC_R(m)$ holds, and $\mcC_R(m)$ is a ringset by Lemma \ref{lem:basic ringsets}.
\end{proof}

Any finite subset $S \subseteq R$ can be partitioned as $S=\bigcup_{i=1}^t S_i$, where each subset $S_i$ lies in a different minimal polynomial class. In Corollary \ref{cor:min poly decomp}, we prove that $S$ is a ringset if and only if each $S_i$ is a ringset. Thus, describing all finite ringsets of $R$ amounts to classifying the subsets of $\mcC_R(m)$ that are ringsets. For this, we introduce the notion of a \textit{reduced} set.

\begin{Def}\label{def:reduced set}
Let $S \subseteq R$. For $a \in \Z$, let $S+a := \{\alpha+a \mid \alpha \in S\}$. We say that $S$ is \textit{reduced} if both of the following two conditions hold.
\begin{enumerate}[(i)]
\item $S \subseteq \mcC_R(m)$ for some monic quadratic $m \in \Z[x]$.
\item For all $a \in \Z$ and all $n \in \Z$, $n \geq 2$, $S+a \not\subseteq nR$.
\end{enumerate}
\end{Def}

Note that a reduced set is necessarily finite by condition (i) of Definition \ref{def:reduced set} and Remark \ref{C_R(m) is finite}. More details and motivation for the definition of a reduced set are given in Section \ref{sec:min poly reduced}. We prove in Proposition \ref{prop:reduced set exists} that each nonempty $S \subseteq \mcC_R(m)$ has an associated reduced set $T$, and furthermore $S$ is a ringset if and only if $T$ is a ringset. This leads us to our main result, which is a full classification of the reduced subsets of $R$ that are ringsets.

\begin{Thm}\label{thm:main}
Let $S \subsetneq R$ be reduced and such that $|S| \geq 2$. Let $\Delta(S) = \{\alpha-\beta \mid \alpha, \beta \in S\}$ and let $\Gamma(S) = \gcd(\{N(\delta) \mid \delta \in \Delta(S)\})$.
\begin{enumerate}[(1)]
\item Assume $R = \LQ$. Then, $S$ is a ringset if and only if $\Gamma(S)=2$; or $\Gamma(S)=4$; or $\Gamma(S)=8$ and the following condition holds: there exist $\delta_1, \delta_2 \in \Delta(S)$ such that $\delta_1$ and $\delta_2$ are congruent modulo 4 to distinct residues in $\{2\bfi+2\bfj, \; 2\bfi+2\bfk, \; 2\bfj+2\bfk\}$.
\item Assume $R = \HQ$. Then, $S$ is a ringset if and only if $\Gamma(S)=1$ or $\Gamma(S) = 2$.
\end{enumerate}
\end{Thm}

The majority of this paper is dedicated to proving Theorem \ref{thm:main} and working with the quantity $\Gamma(S)$ that is given in its statement. Section \ref{sec:reduced ringsets} deals with most cases in Theorem \ref{thm:main} (see Theorem \ref{thm:all except 8}), while Section \ref{sec:Gamma is 8} handles the exceptional case $R = \LQ$ and $\Gamma(S)=8$. We close the paper with a short summary of how to use our results to decide if a finite subset of $R$ is a ringset, and give some remarks on the difficulty of determining whether an infinite subset of $R$ is a ringset.

\section{Minimal Polynomial Classes and Reduced Sets}\label{sec:min poly reduced}

Recall from the introduction that $R$ denotes either the Lipschitz quaternions $\LQ$ or the Hurwitz quaternions $\HQ$. Any result that does not specify $R = \LQ$ or $R = \HQ$ will apply to both rings. For a monic $m \in \Z[x]$ of degree 1 or 2, $\mcC_R(m)$ is the set of elements of $R$ having minimal polynomial $m$. Our first theorem shows that distinct minimal polynomial classes can be separated via polynomials from $\Q[x]$. As a consequence of this, it will be enough to work with subsets of the classes $\mcC_R(m)$ when studying finite ringsets of $R$.

Note that when $f \in \D[x]$ and $g \in \Q[x]$, we have $(gf)(\alpha) = (fg)(\alpha) = f(\alpha)g(\alpha)$ for all $\alpha \in \D$. We will use this fact freely and frequently throughout the paper.

\begin{Thm}\label{thm:separation}
Let $S \subseteq R$ be finite and nonempty. Let $T \subseteq R$ be nonempty and such that each element of $T$ has the same minimal polynomial $m \in \Z[x]$, but no element of $S$ has minimal polynomial $m$. That is, $T \subseteq \mcC_R(m)$ and $S \cap \mcC_R(m) = \varnothing$.
\begin{enumerate}[(1)]
\item There exists $F \in \Q[x]$ such that for $\alpha \in S \cup T$,
\begin{equation*}
F(\alpha) = \begin{cases} 0, & \alpha \in S\\ 1, & \alpha \in T. \end{cases}
\end{equation*}
\item If $T$ is not a ringset, then $S \cup T$ is not a ringset.
\end{enumerate}
\end{Thm}
\begin{proof}
(1) Let $M$ be the product of all the distinct minimal polynomials of elements in $S$. Then, $M$ is monic and has integer coefficients. Note that $m$ does not divide $M$ in $\Q[x]$, because no element of $S$ has minimal polynomial $m$, and $\mu_\alpha$ is irreducible over $\Q[x]$ for all $\alpha \in \D$. If $\deg m = 1$, then $T = \{\alpha\}$ for some $\alpha \in \Z$. In this case, we can take $F(x) = M(x)/M(\alpha)$, and we are done.

From here, assume that $\deg m = 2$. Divide $M$ by $m$ to get $M = qm+r$, where $q,r \in \Z[x]$ and either $r = 0$ or $\deg r \leq 1$. Then, $M(\alpha) = 0$ for all $\alpha \in S$, and $M(\alpha) = r(\alpha)$ for all $\alpha \in T$. Since $m$ does not divide $M$, $M(\alpha) \ne 0$ and $r \ne 0$.

Let $m(x) = x^2 - 2ax+n$. Since $T \subseteq \mcC_R(m)$, each element of $T$ has real part $a$ and norm $n$. Write $r(x) = c_1 x + c_0$ for some $c_0, c_1 \in \Z$. For every $\alpha \in T$, we have
\begin{equation*}
r(\alpha) = c_1 \alpha + c_0 = (c_1a + c_0) + c_1(\alpha - a).
\end{equation*}
Let $d=N(r(\alpha)) = c_0^2 + 2c_0c_1a + c_1^2n$. Observe that $d$ is independent of the choice of $\alpha \in T$. Moreover, $d \ne 0$ because $r(\alpha) \ne 0$, and $0$ itself is the only element of $\D$ with norm 0. Let $s(x) = c_1a + c_0 - c_1(x-a) \in \Z[x]$. Then, $s(\alpha) = \overline{r(\alpha)}$ for all $\alpha \in T$. Since both $M$ and $s$ have integer coefficients, $(Ms)(\alpha) = M(\alpha)s(\alpha)$ for all $\alpha \in S \cup T$. Thus, $(Ms)(\alpha) = 0$ when $\alpha \in S$, and 
\begin{equation*}
(Ms)(\alpha) = M(\alpha)s(\alpha) = r(\alpha)s(\alpha) = d
\end{equation*}
when $\alpha \in T$. Hence, the polynomial $F(x) = (Ms)(x)/d$ has the desired properties.

(2) Assume that $T$ is not a ringset. Then, there exist $f, g \in \Int(T,R)$ such that $fg \notin \Int(T,R)$. Let $F \in \Q[x]$ be as in part (1). Then, for any $\alpha \in S \cup T$, 
\begin{equation*}
(fF)(\alpha) = f(\alpha)F(\alpha) \text{ and } (gF)(\alpha) = g(\alpha)F(\alpha),
\end{equation*}
so both $fF$ and $gF$ are in $\Int(S \cup T, R)$. We claim that $fF \cdot gF \notin \Int(S \cup T, R)$. To see this, pick $\beta \in T$ such that $(fg)(\beta) \notin R$. We have 
\begin{equation*}
(fF \cdot gF)(\beta) = (fg)(\beta)\cdot (F(\beta))^2 = (fg)(\beta) \notin R.
\end{equation*}
Thus, $fF \cdot gF \notin \Int(S \cup T, R)$ and $S \cup T$ is not a ringset.
\end{proof}

\begin{Cor}\label{cor:min poly decomp}
Let $S \subseteq R$ be finite and nonempty. Let $m_1, \ldots, m_t$ be the distinct minimal polynomials of all of the elements of $S$. For each $1 \leq i \leq t$, let $S_i = S \cap \mcC_R(m_i)$. Then, $S$ is a ringset if and only if $S_i$ is a ringset for all $1 \leq i \leq t$.
\end{Cor}
\begin{proof}
The sets $S_1, \ldots, S_t$ form a partition of $S$, so $S = \bigcup_{i=1}^t S_i$. If each $S_i$ is a ringset, then $S$ is a ringset by Lemma \ref{lem:basic ringsets}. Conversely, if some $S_i$ is not a ringset, then $S$ is not a ringset by Theorem \ref{thm:separation}.
\end{proof}

By Corollary \ref{cor:min poly decomp}, determining the finite ringsets of $R$ amounts to characterizing the ringsets within each minimal polynomial class $\mcC_R(m)$. When $m$ is linear, $\mcC_R(m) \subseteq \Z$ and every subset of $\mcC_R(m)$ is a ringset. So, for the remainder of the paper, we will focus on the classes $\mcC_R(m)$ with $m$ quadratic. 

\begin{Rem}\label{Disjoint classes}
Note that when $R = \HQ$, the minimal polynomial classes that lie within $\LQ$ are disjoint from those within $\HQ\setminus\LQ$. Explicitly, $\mcC_R(m) \subseteq \LQ$ if and only if the linear coefficient of $m$ is even, and $\mcC_R(m) \subseteq \HQ\setminus\LQ$ if and only if the linear coefficient of $m$ is odd. This dichotomy will be used several times to split proofs about a set $S \subseteq \mcC_R(m)$ into two cases: one case in which $S \subseteq \HQ\setminus\LQ$, and a second case in which $S \subseteq \LQ$.
\end{Rem}

For a subset $S$ of a single minimal polynomial class $\mcC_R(m)$, we can further reduce the problem by considering integer translations and integer multiples of $S$. When $a, n \in \Z$ and $n \ne 0$, we define $S+a:=\{\alpha+a \mid \alpha \in S\}$ and $nS:=\{n\alpha \mid \alpha \in S\}$.

\begin{Lem}\label{lem:translation}
Let $S \subseteq R$, $a \in \Z$, and $n \in \Z$ with $n \ne 0$. The following are equivalent:
\begin{enumerate}[(i)]
\item $S$ is a ringset.
\item $S+a$ is a ringset.
\item $nS$ is a ringset.
\end{enumerate}
\end{Lem}
\begin{proof}
Note that since $a, n \in \Z$, both $S+a$ and $nS$ are subsets of $R$. Since $a$ is central in $\D$, the mapping $f(x) \mapsto f(x-a)$ is a ring automorphism on $\D[x]$. The image of $\Int(S,R)$ under this mapping is $\Int(S+a,R)$; hence, $S$ is a ringset if and only if $S+a$ is a ringset. The analogous result holds for $S$ and $nS$ by considering the mapping $f(x) \mapsto f(x/n)$, which is also a ring automorphism on $\D[x]$.
\end{proof}

\begin{Ex}\label{ex:4+5i}
Let $S = \{4+5\bfi, 4+5\bfj\} \subseteq \mcC_R(x^2-8x+41)$ and $T = \{\bfi, \bfj\} \subseteq \mcC_R(x^2+1)$. Then, $S = 4+5T$. In light of Lemma \ref{lem:translation}, $S$ is a ringset if and only if $T$ is a ringset. As our forthcoming work will show (see Theorem \ref{thm:Gamma is 1, 2, or 4}), we can conclude that $T$ is a ringset simply because $N(\bfi-\bfj)=2$. Hence, $S$ is also a ringset.
\end{Ex}

The relationships among $S$, $a+S$, and $nS$ are the inspiration for reduced subsets of $R$. Recall from Definition \ref{def:reduced set} that $S \subseteq R$ is \textit{reduced} if the following two conditions hold:
\begin{enumerate}[(i)]
\item $S \subseteq \mcC_R(m)$ for some monic quadratic $m \in \Z[x]$. 
\item For all $a \in \Z$ and all $n \in \Z$, $n \geq 2$, $S+a \not\subseteq nR$.
\end{enumerate}
Essentially, reduced sets are those for which we have ``factored out'' as many positive integers as possible from the imaginary coefficients of the elements of $S$. We state this more precisely in Proposition \ref{prop:reduced equivalent def} below. Note that since $\mcC_R(m)$ is always finite (see Remark \ref{C_R(m) is finite}), any reduced set is finite by condition (i). Also by condition (i), $S \cap \Z = \varnothing$, and condition (ii) implies that $S$ is nonempty. So, any reduced set is a finite nonempty subset of $R \setminus \Z$. Finally, each element of a reduced set has the same real part. We will exploit these facts often in our subsequent work.

\begin{Ex}\label{ex:reduced sets}
We give some examples to illustrate reduced sets.
\begin{itemize}
\item With $S$ and $T$ as in Example \ref{ex:4+5i}, $S$ is not reduced, but $T$ is reduced.

\item Let $S = \{3\bfi, 3\bfj, 3\bfk\}$ and $T = \{3\bfi, 3\bfj, \bfi+2\bfj+2\bfk\}$, which are both subsets of $\mcC_R(x^2+9)$. Then, $S$ is not reduced, because $S \subseteq 3R$. However, $T$ is reduced.

\item Let $R = \HQ$. Let $S = \{(1+5\bfi+15\bfj+25\bfk)/2,  (1-5\bfi-15\bfj-25\bfk)/2\} \subseteq \mcC_R(x^2-x+219)$ and $T = \{(1+\bfi+3\bfj+5\bfk)/2,  (1-\bfi-3\bfj-5\bfk)/2\} \subseteq \mcC_R(x^2-x+9)$. Then, $T$ is reduced, but $S+2=5T$ and hence $S$ is not reduced.

\item Determining whether a subset is reduced or not can depend on whether $R = \LQ$ or $R = \HQ$. For instance, let $S = \{\bfi+\bfj+\bfk, \bfi-\bfj-\bfk\}$. Then, $S$ is reduced with respect to $\LQ$. However, $\HQ$ contains both $(1+\bfi+\bfj+\bfk)/2$ and $(1+\bfi-\bfj-\bfk)/2$, so $S+1 \subseteq 2\HQ$. Thus, $S$ is not reduced with respect to $\HQ$.

\item More generally, when $R = \HQ$, $S \subseteq \mcC_R(m) \subseteq \LQ$ for some monic quadratic $m$, and each imaginary coefficient of every element of $S$ is odd, then $S$ is not reduced. In this case, either $S$ or $S+1$ is in $2\HQ$, depending on whether the real part of the elements of $S$ is odd or even.
\end{itemize}
\end{Ex}

\begin{Prop}\label{prop:reduced equivalent def}
Let $S \subsetneq R$ be nonempty. Assume that $S \subseteq \mcC_R(m)$ for some monic quadratic $m \in \Z[x]$.
\begin{enumerate}[(1)]
\item Assume $R = \LQ$. Then, $S$ is reduced if and only if for each prime $p$, there exists $\beta = b_0 + b_1\bfi + b_2\bfj + b_3\bfk \in S$ such that $p \nmid b_i$ for some $1 \leq i \leq 3$.

\item Assume $R = \HQ$. By Remark \ref{Disjoint classes}, either $S \subsetneq \HQ\setminus\LQ$ or $S \subsetneq \LQ$.
\begin{enumerate}[(a)]
\item Assume $S \subsetneq \HQ\setminus\LQ$. Then, $S$ is reduced if and only if for each odd prime $p$, there exists $\beta = (b_0 + b_1\bfi + b_2\bfj + b_3\bfk)/2 \in S$ such that $p \nmid b_i$ for some $1 \leq i \leq 3$.

\item Assume $S \subsetneq \LQ$. Then, $S$ is reduced if and only if both of the following conditions hold:
\begin{enumerate}[(i)]
\item $S$ is reduced with respect to $\LQ$.
\item There exists $\alpha \in S$ such that at least one imaginary coefficient of $\alpha$ is even.
\end{enumerate}
\end{enumerate}
\end{enumerate}
\end{Prop}
\begin{proof}
For each part, we prove the contrapositive statement.

(1) $(\Rightarrow)$ Assume there exists a prime $p$ such that $p$ divides every imaginary coefficient of every element of $S$. Let $b_0$ be the real part of each element of $S$. Then, $S - b_0 \subseteq pR$, so $S$ is not reduced.

$(\Leftarrow)$ Assume that $S$ is not reduced. Then, there exist $a, n \in \Z$ with $n \geq 2$ and such that $S+a \subseteq nR$. Let $p$ be a prime divisor of $n$. Then, $S+a \subseteq pR$, which means that $p$ divides each imaginary coefficient of every element of $S$.

(2a) This is similar to the proof of (1). First, assume there exists a prime $p$ such that for every $(b_0 + b_1\bfi + b_2\bfj + b_3\bfk)/2 \in S$, $p$ divides $b_i$, $1 \leq i \leq 3$. Necessarily, $p$ must be odd. Let $a \in \Z$ such that $b_0 + 2a = p$. Then, $S+a \subseteq pR$ and $S$ is not reduced.

Conversely, assume that $S$ is not reduced. As in the proof of (1), $S+a \subseteq pR$ for some $a \in \Z$ and some prime $p$. Given $\beta = (b_0 + b_1\bfi + b_2\bfj + b_3\bfk)/2 \in S$, we have 
\begin{equation*}
b_0 + 2a + b_1\bfi + b_2\bfj + b_3\bfk = 2(\beta+a) \in 2pR.
\end{equation*}
Let $\gamma = c_0+c_1\bfi+c_2\bfj+c_3\bfk \in R$ be such that $2(\beta+a) = 2p\gamma$. Then, $b_i = 2pc_i$ for all $1 \leq i \leq 3$. Since $2c_i \in \Z$, this shows that $p \mid b_i$ for each $1 \leq i \leq 3$, and $p$ must be odd because $\beta \in \HQ\setminus\LQ$.

(2b) $(\Rightarrow)$ Certainly, if $S$ is not reduced with respect to $\LQ$, then $S$ is not reduced with respect to $\HQ$. Also, as pointed out at the end of Example \ref{ex:reduced sets}, if every imaginary coefficient of every element of $S$ is odd, then $S$ is not reduced with respect to $\HQ$.

$(\Leftarrow)$ Assume that $S$ is not reduced with respect to $\HQ$. If $S$ is not reduced with respect to $\LQ$, then the negation of (i) holds and we are done. So, assume that $S$ is reduced with respect to $\LQ$. Find $a \in \Z$ and a prime $p$ such that $S+a \subseteq pR$. By (1), there exists $\beta = b_0+b_1\bfi+b_2\bfj+b_3\bfk \in S$ such that $p \nmid b_i$ for some $1 \leq i \leq 3$. Without loss of generality, assume that $p \nmid b_1$.

Let $\gamma = c_0+c_1\bfi+c_2\bfj+c_3\bfk \in R$ such that $\beta+a=p\gamma$. Then, $b_1=pc_1$. Since $p$ does not divide $b_1$, the only way this can be true is if $p=2$ and $c_1 \in \Z+\tfrac{1}{2}$. Consequently, $c_i \in \Z+\tfrac{1}{2}$ for all $0 \leq i\leq 3$, and so $b_i = 2c_i$ is odd for all $1 \leq i \leq 3$. 

Finally, let $\alpha = b_0+a_1\bfi+a_2\bfj+a_3\bfk \in S$. Then, $N(\alpha)=N(\beta)$. Considering these norms modulo 4 and recalling that $b_i$ is odd for $1 \leq i \leq 3$, we have
\begin{equation*}
a_1^2 + a_2^2 + a_3^2 \equiv b_1^2 + b_2^2 + b_3^2 \equiv 3 \pmod{4}.
\end{equation*}
This shows that $a_i$ is odd for all $1 \leq i \leq 3$. Since $\alpha \in S$ was arbitrary, the negation of (ii) holds.
\end{proof}

We record as a corollary the portions of Proposition \ref{prop:reduced equivalent def} that we will reference most often.

\begin{Cor}\label{cor:prime factors}
Let $S \subsetneq R$ be reduced. By Remark \ref{Disjoint classes}, either $S \subsetneq \HQ\setminus\LQ$ or $S \subsetneq \LQ$.
\begin{enumerate}[(1)]
\item If $S \subsetneq \LQ$, then for each prime $p$, there exists $\beta = b_0 + b_1 \bfi + b_2 \bfj + b_3\bfk \in S$ such that $p \nmid b_i$ for some $1 \leq i \leq 3$.
\item If $S \subsetneq \HQ\setminus\LQ$, then for each odd prime $p$, there exists $\beta = (b_0 + b_1 \bfi + b_2 \bfj + b_3\bfk)/2 \in S$ such that $p \nmid b_i$ for some $1 \leq i \leq 3$.
\end{enumerate}
\end{Cor}

Regardless of whether we work over $\LQ$ or $\HQ$, it suffices to consider reduced sets when looking for ringsets within minimal polynomial classes. In order to prove this---and to help decide when reduced sets are ringsets---we need one more definition.

\begin{Def}\label{def:Gamma(S)}
Let $S \subseteq R$ be nonempty. We define $\Delta(S) := \{\alpha - \beta \mid \alpha, \beta \in S \}$. When $|S| \geq 2$, we set $\Gamma(S) := \gcd(\{N(\delta) \mid \delta \in \Delta(S)\})$. If $S$ is a singleton set, then $\Delta(S) =\{0\}$. In this case, we define $\Gamma(S) := 0$ to avoid complications with the definition of $\gcd(\{0\})$.
\end{Def}

\begin{Rem}\label{rem:Gamma}
When $S \subseteq \LQ$ is reduced, $\Gamma(S)$ will always be even. Indeed, let $\alpha = a_0 + a_1\bfi + a_2\bfj + a_3\bfk$ and $\beta = a_0 + b_1\bfi + b_2\bfj + b_3\bfk$ be elements of $S$. Then, $a_1^2+a_2^2+a_3^2 = b_1^2+b_2^2+b_3^2$, so
\begin{equation*}
N(\alpha-\beta) = (a_1-b_1)^2+(a_2-b_2)^2+(a_3-b_3)^2 = 2N(\alpha-a) - 2(a_1b_1 + a_2b_2 + a_3b_3).
\end{equation*}
In contrast, $\Gamma(S)$ can be odd if $S \subseteq \HQ\setminus\LQ$, as demonstrated by $S = \{(1+\bfi+\bfj+\bfk)/2, (1+\bfi+\bfj-\bfk)/2\}$ with $\Gamma(S)=1$.
\end{Rem}

The quantity $\Gamma(S)$ is what we will ultimately use to classify reduced subsets of $R$ that are ringsets. We pause to give some motivation for this definition. When $S \subseteq \mcC_R(m)$ and $f \in \Int(S,R)$, we can divide $f$ by $m$ to get $f=qm+r$, where $q,r \in \D[x]$ and the remainder polynomial $r$ is in $\Int(S,R)$ and has the form $r(x)=\gamma_1x + \gamma_0$. As we show in Section \ref{sec:reduced ringsets}, we can decide whether $S$ is a ringset by studying the polynomials $r$. The subjects of Definition \ref{def:Gamma(S)} arise in the following way. Given $\alpha, \beta \in S$, we have $f(\alpha)-f(\beta) = r(\alpha)-r(\beta) = \gamma_1(\alpha-\beta)$. Considering all such differences as $\alpha$ and $\beta$ run through $S$ leads to the set $\Delta(S)$, and then to $\Gamma(S)$.

Our first use of $\Gamma(S)$ is to prove that any nonempty subset of a minimal polynomial class has an associated reduced subset.

\begin{Prop}\label{prop:reduced set exists}
Let $S \subsetneq R$ be nonempty. Assume that $S \subseteq \mcC_R(m)$ for some monic quadratic $m \in \Z[x]$. Then, there exists a reduced set $T \subsetneq R$ such that $S = a+nT$ for some $a, n \in \Z$ with $n \geq 1$. Moreover, $S$ is a ringset if and only if $T$ is a ringset.
\end{Prop}
\begin{proof}
First, consider the case where $S=\{\alpha\}$ is a singleton set. If $\alpha \in \LQ$, then let $\alpha = a_0 + a_1 \bfi + a_2 \bfj + a_3 \bfk$, where each $a_i \in \Z$. Let $n=\gcd(a_1,a_2,a_3)$ and $\beta = (\alpha-a_0)/n$. If $R = \LQ$, then $\{\beta\}$ is reduced. We can take $T =\{\beta\}$ and then $S = a_0 + nT$. If $R = \HQ$, then write $\beta = b_1\bfi+b_2\bfj+b_3\bfk$ for some $b_1, b_2, b_3 \in \Z$. Necessarily, $\gcd(b_1,b_2,b_3)=1$. If some $b_i$ is even, then $\{\beta\}$ is reduced and we can again take $T=\{\beta\}$. On the other hand, if each $b_i$ is odd, then $(1+\beta)/2 \in \HQ$ and $\{(1+\beta)/2\}$ is reduced. In this case, we take $T = \{(1+\beta)/2\}$ and then $S = a_0-n + 2nT$.

Next, assume that $S = \{\alpha\}$ and $\alpha \in \HQ\setminus\LQ$. Then, $\alpha = (a_0+a_1\bfi+a_2\bfj+a_3\bfk)/2$ for some odd integers $a_i$, $0 \leq i \leq 3$. Let $n=\gcd(a_1,a_2,a_3)$, which is odd. There exists $a \in \Z$ such that $2a+a_0=n$. Let $\beta = (\alpha+a)/n = (1+(a_1/n)\bfi+(a_2/n)\bfj+(a_3/n)\bfk)/2$. Then, $T = \{\beta\}$ is reduced and $S=-a+nT$.

From here, assume that $|S| \geq 2$. If $S$ is reduced, then we are done. If not, then $S = a_1+n_1T_1$ for some $T_1 \subseteq R$ and $a_1,n_1 \in \Z$ with $n_1 \geq 2$. By construction, each element of $T_1$ has the same minimal polynomial. Observe that $n_1$ divides each element of $\Delta(S)$, and in fact $\Delta(S) = n_1 \Delta(T_1)$. It follows that $\Gamma(S) = n_1^2 \Gamma(T_1)$. Since both $\Gamma(S)$ and $\Gamma(T_1)$ are positive integers and $n_1 \geq 2$, $\Gamma(S) > \Gamma(T_1)$. If $T_1$ is reduced, then we are done. If not, then we may repeat the above argument with $T_1$ to get $T_1 = a_2 + n_2 T_2$ for some $T_2 \subseteq R$ and $a_2, n_2 \in \Z$. This procedure cannot continue indefinitely, because the sequence $\{\Gamma(S)$, $\Gamma(T_1)$, $\Gamma(T_2)$, $\ldots\}$ is strictly decreasing. Thus, after a finite number of steps, we can express $S$ as $S = a +nT$, where $T$ is reduced. The final claim follows from Lemma \ref{lem:translation}.
\end{proof}

\section{Classifying Reduced Sets that are Ringsets}\label{sec:reduced ringsets}

In this section, we will use $\Gamma(S)$ to determine when a reduced subset $S$ of $R$ is a ringset. Our general strategy is to examine the linear polynomials in $\Int(S,R)$, and consider their effect not just on elements of $S$, but on the set of all conjugates of elements of $S$.

\begin{Def}\label{def:S star}
For $S \subseteq R$, we define $S^\ast:=\{u \alpha u^{-1} \mid \alpha \in S, u \in R^\times\}$.
\end{Def}

\begin{Lem}\label{lem:difference of conjugates}
Let $\alpha \in R$ and $u \in R^\times$. 
\begin{enumerate}[(1)]
\item Assume $R = \LQ$. Then, $\alpha - u \alpha u^{-1} \in 2R$. 
\item Assume $R = \HQ$ and let $I$ be the two-sided ideal of $R$ generated by $1+\bfi$. Then, $\alpha - u \alpha u^{-1} \in I$. Moreover, if $\beta \in R$ and $N(\beta)$ is even, then $\beta \in I$.
\end{enumerate}
\end{Lem}
\begin{proof}
When $R = \LQ$, $R/2R$ is a commutative ring of order 16. In this residue ring, $\alpha - u \alpha u^{-1} \equiv \alpha - \alpha \equiv 0$. Similarly, when $R = \HQ$, $R/I \cong \F_4$ is commutative, and hence $\alpha - u \alpha u^{-1} \in I$. In this case, if $\beta \in R$ and $N(\beta)$ is even, then $\beta \cdot \overline{\beta} \equiv 0$ in $R/I$. Since $R/I$ is a field, this means that $\beta \in I$.
\end{proof}

\begin{Lem}\label{lem:S star and linear polys}
Let $S \subsetneq R$ be nonempty. Assume that $S \subseteq \mcC_R(m)$ for some monic quadratic $m \in \Z[x]$.
\begin{enumerate}[(1)]
\item If every linear polynomial in $\Int(S,R)$ is in $\Int(S^\ast, R)$, then $S$ is a ringset.

\item Let $r(x) = \gamma_1 x + \gamma_0 \in \Int(S,R)$.
\begin{enumerate}[(a)]
\item Assume $R=\LQ$. If $2 \gamma_1 \in R$, then $r \in \Int(S^\ast,R)$.
\item Assume $R=\HQ$. If $\gamma_1(1+\bfi) \in R$, then $r \in \Int(S^\ast,R)$.
\end{enumerate}
\end{enumerate}
\end{Lem}
\begin{proof}
(1) Assume that the stated condition holds. Note that $S^\ast$ is a union of full conjugacy classes in $R$, hence is a ringset by Lemma \ref{lem:basic ringsets}. We will show that $\Int(S,R) = \Int(S^\ast,R)$. Certainly, $\Int(S^\ast,R) \subseteq \Int(S,R)$ because $S \subseteq S^\ast$. For the reverse inclusion, let $f \in \Int(S,R)$. Divide $f$ by $m$ to get $f=qm+r$, where $q, r \in \D[x]$ and $r(x) = \gamma_1 x+ \gamma_0$ for some $\gamma_1, \gamma_0 \in \D$. Note that each element of $S^\ast$ has minimal polynomial $m$, so $f(\beta) = r(\beta)$ for all $\beta \in S^\ast$. In particular, $r \in \Int(S,R)$. Now, if $\gamma_1 = 0$, then $\gamma_0 \in R$ and hence $f \in \Int(S^\ast,R)$. Otherwise, $r$ is a linear polynomial in $\Int(S,R)$, and thus is in $\Int(S^\ast,R)$ by assumption. In this case, $f$ is once again in $\Int(S^\ast,R)$.

(2a) Assume that $2\gamma_1 \in R$. Let $\beta \in S^\ast$. Then, $\beta = u\alpha u^{-1}$ for some $\alpha \in S$ and some $u \in R^\times$. By Lemma \ref{lem:difference of conjugates}, $\alpha - \beta \in 2R$. Thus,
\begin{equation*}
r(\alpha) - r(\beta) = \gamma_1(\alpha - \beta) \in R.
\end{equation*}
By assumption, $r(\alpha) \in R$, so $r(\beta) \in R$ as well. Thus, $r \in \Int(S^\ast,R)$.

(2b) Let $I$ be the two-sided ideal of $\HQ$ generated by $1+\bfi$. It is well known that $1+\bfi$ generates $I$ as both a left ideal and a right ideal of $\HQ$. So, if $\varepsilon \in I$, then there exist $\varepsilon_1, \varepsilon_2 \in R$ such that $\varepsilon = \varepsilon_1(1+\bfi) = (1+\bfi)\varepsilon_2$. Proceeding as in part (2a), we can show that $\alpha - \beta \in (1+\bfi)R$, and hence $\gamma_1(\alpha-\beta) \in R$. The result follows.
\end{proof}

From here, we will use the quantity $\Gamma(S)$ to characterize reduced sets $S$ that are ringsets. We begin with some positive results.

\begin{Thm}\label{thm:Gamma is 1, 2, or 4}
Let $S \subsetneq R$ be reduced.
\begin{enumerate}[(1)]
\item If $\Gamma(S) = 1$, then $S$ is a ringset.
\item If $\Gamma(S) = 2$, then $S$ is a ringset.
\item If $R = \LQ$ and $\Gamma(S) = 4$, then $S$ is a ringset.
\end{enumerate}
\end{Thm}
\begin{proof}
For each part, we use the fact that $\Gamma(S)$ can be written as a linear combination of the norms of differences of elements of $S$. So, throughout we will assume that there exists an integer $t \geq 1$ and, for each $1 \leq i \leq t$, elements $\alpha_i, \beta_i \in S$ and integers $n_i$ such that
\begin{equation}\label{Gamma linear combo}
\Gamma(S) = \sum_{i=1}^t n_iN(\alpha_i-\beta_i).
\end{equation}

Next, let $r(x) = \gamma_1 x + \gamma_0 \in \Int(S,R)$. Then, for every $\alpha, \beta \in S$, $\gamma_1(\alpha-\beta) = r(\alpha)-r(\beta) \in R$. Thus, $N(\alpha - \beta)\gamma_1 = \gamma_1(\alpha - \beta)\big(\,\overline{\alpha-\beta}\,\big) \in R$. Applying \eqref{Gamma linear combo}, we obtain
\begin{equation}\label{gamma linear combo}
\Gamma(S)\gamma_1 = \sum_{i=1}^t n_i\gamma_1(\alpha_i - \beta_i)\big(\,\overline{\alpha_i-\beta_i}\,\big) \in R.
\end{equation}
Note that \eqref{gamma linear combo} holds for the leading coefficient $\gamma_1$ of any linear polynomial in $\Int(S,R)$.

(1) Assume that $\Gamma(S)=1$. By \eqref{gamma linear combo}, $\gamma_1 \in R$. Since $r \in \Int(S,R)$, this forces $\gamma_0 \in R$, and hence $r \in R[x]$. Consequently, $r \in \Int(S^\ast,R)$ and $S$ is a ringset by Lemma \ref{lem:S star and linear polys}.

(2) Assume that $\Gamma(S) = 2$. Then, $2\gamma_1 \in R$ by \eqref{gamma linear combo}. If $R = \LQ$, then $S$ is a ringset by Lemma \ref{lem:S star and linear polys}. So, assume that $R = \HQ$. We will use \eqref{gamma linear combo} to show that $\gamma_1(1+\bfi) \in R$. 

For each $1 \leq i \leq t$, the norm of $\alpha_i-\beta_i$ is even, so $\alpha_i - \beta_i \in R(1+\bfi)$ by Lemma \ref{lem:difference of conjugates}. For each $i$, let $\varepsilon_i \in R$ be such that $\alpha_i-\beta_i = \varepsilon_i(1+\bfi)$. Also, note that since the real part of $\alpha_i-\beta_i$ is 0, $\overline{\alpha_i-\beta_i} = -(\alpha_i-\beta_i) = -\varepsilon_i(1+\bfi)$. From \eqref{gamma linear combo}, we obtain
\begin{equation}\label{eq:1+i}
\gamma_1(1+\bfi)(1-\bfi) = 2\gamma_1 = \sum_{i=1}^t n_i\gamma_1(\alpha_i - \beta_i)(-\varepsilon_i)(1+\bfi).
\end{equation}
The inverse of $(1-\bfi)$ is $(1-\bfi)^{-1} = (1+\bfi)/2$. Multiplying each side of \eqref{eq:1+i} on the right by this element produces
\begin{align*}
\gamma_1(1+\bfi) &= \sum_{i=1}^t n_i\gamma_1(\alpha_i - \beta_i)(-\varepsilon_i)(1+\bfi)(1+\bfi)/2\\
&= \sum_{i=1}^t n_i\gamma_1(\alpha_i - \beta_i)(-\varepsilon_i)(\bfi).
\end{align*}
This last expression is in $R$ because $\gamma_1(\alpha_i-\beta_i) \in R$ for each $i$. We conclude that $\gamma_1(1+\bfi) \in R$. Hence, $S$ is a ringset by Lemma \ref{lem:S star and linear polys}.

(3) Assume that $R = \LQ$ and $\Gamma(S) = 4$. Then, $N(\delta) \equiv 0 \pmod{4}$ for all $\delta \in \Delta(S)$. Given $\delta = d_1\bfi + d_2\bfj + d_3\bfk$, where each $d_i \in \Z$, this means that $d_1^2 + d_2^2 + d_3^2 \equiv 0 \pmod{4}$. Thus, each $d_i$ is even. Consequently, $\alpha - \beta \in 2\LQ$ and $\overline{\alpha-\beta} \in 2\LQ$ for all $\alpha, \beta \in S$. Applying \eqref{gamma linear combo}, we have
\begin{equation}\label{eq:Gamma=4}
2\gamma_1 = \tfrac{1}{2}(4\gamma_1) = \sum_{i=1}^t n_i\gamma_1(\alpha_i - \beta_i)\Big(\dfrac{\;\overline{\alpha_i-\beta_i}\;}{2}\Big).
\end{equation}
For each $i$, both $\gamma_1(\alpha_i - \beta_i)$ and $\big(\,\overline{\alpha_i-\beta_i}\,\big)/2$ are in $R$. So, $2\gamma_1 \in R$, and hence $S$ is a ringset by Lemma \ref{lem:S star and linear polys}.
\end{proof}

As Theorem \ref{thm:Gamma is 1, 2, or 4} shows, $S$ will be a ringset when $\Gamma(S)$ is a small power of 2. We next demonstrate that $S$ will not be a ringset if $\Gamma(S)$ is too large. To prove this, we will work with elements of the form $u\alpha - \alpha u$, where $\alpha \in R$ and $u \in R^\times$. As an aid to the reader, we record the following computations. Let $\alpha = a_0 + a_1\bfi + a_2\bfj + a_3\bfk$. Then,
\begin{align*}
\bfi \alpha - \alpha\bfi &= -2a_3\bfj + 2a_2\bfk, \quad \bfj \alpha - \alpha\bfj = 2a_3\bfi - 2a_1\bfk, \quad \bfk\alpha - \alpha\bfk = -2a_2\bfi + 2a_1\bfj, \quad \text{and}\\
\bfh\alpha - \alpha\bfh &= (a_3-a_2)\bfi + (a_1-a_3)\bfj + (a_2-a_1)\bfk.
\end{align*}

\begin{Prop}\label{prop:powers of 2 divide Gamma}
Let $S \subsetneq R$ be reduced.
\begin{enumerate}[(1)]
\item Assume $R = \LQ$. If $16$ divides $\Gamma(S)$, then $S$ is not a ringset.
\item Assume $R = \HQ$. If $4$ divides $\Gamma(S)$, then $S$ is not a ringset.
\end{enumerate}
\end{Prop}
\begin{proof}
By Lemma \ref{lem:basic ringsets}, both parts of the proposition are true when $S$ is a singleton set. So, we will assume throughout that $|S| \geq 2$, and hence that $\Gamma(S) \ne 0$.

(1) Assume that $16 \mid \Gamma(S)$. Then, each $\delta \in \Delta(S)$ has the form $\delta = d_1\bfi + d_2\bfj + d_3\bfk$, where $d_1, d_2, d_3 \in \Z$ and $N(\delta) = d_1^2+d_2^2+d_3^2 \equiv 0 \pmod{16}$. The only way this congruence can be satisfied is if each $d_i$ is divisible by 4. This holds for all $\delta \in \Delta(S)$, which means that the polynomial $(x-\alpha)/4$ is in $\Int(S,R)$ for all $\alpha \in S$. Since $S$ is reduced, by Corollary \ref{cor:prime factors} there exists $\beta = b_0 + b_1\bfi + b_2\bfj + b_3\bfk \in S$ such that at least one of $b_1, b_2$, or $b_3$ is odd. Without loss of generality, assume that $b_1$ is odd. Let $f(x) = (x-\beta)/4$. Then, $(f\bfj)(x) = (\bfj x - \beta \bfj)/4$, and 
\begin{equation*}
(f\bfj)(\beta) = (\bfj \beta - \beta \bfj)/4 = (2b_3\bfi - 2b_1\bfk)/4 \notin R.
\end{equation*}
This shows that $f \bfj \notin \Int(S,R)$, and therefore $S$ is not a ringset. When $b_2$ or $b_3$ is odd, we can reach the same conclusion by examining $(f\bfi)(\beta)$.\\

(2) Assume that $4 \mid \Gamma(S)$. Recall from Remark \ref{Disjoint classes} that because $S$ is reduced, either $S \subsetneq \LQ$ or $S \subsetneq \HQ\setminus\LQ$. This leads us to consider two cases.\\

\noindent\textbf{\underline{Case 1}}: $S \subsetneq \LQ$

As in the proof of Theorem \ref{thm:Gamma is 1, 2, or 4}(3), each $\delta \in \Delta(S)$ has the form $\delta = d_1\bfi+d_2\bfj+d_3\bfk$ with each $d_i$ even. This means that all coefficients of elements of $S$ have matching parities. That is, given $\alpha = a+a_1\bfi+a_2\bfj+a_3\bfk \in S$ and $\beta = a+b_1\bfi+b_2\bfj+b_3\bfk \in S$, we have $a_i \equiv b_i \pmod{2}$ for all $1 \leq i \leq 3$. 

Suppose that each imaginary coefficient $a_i$ is odd for every $\alpha \in S$. Then, either $S \subseteq 2\HQ$ (if $a$ is odd) or $S+1 \subseteq 2\HQ$ (if $a$ is even). This is impossible, because $S$ is reduced. So, there exists $\beta = a+b_1\bfi+b_2\bfj+b_3\bfk \in S$ such that $b_i$ is even for some $1 \leq i \leq 3$. Furthermore, some $b_j$ must be odd. If not, then---since coefficients of elements of $S$ have matching parities---this would violate Corollary \ref{cor:prime factors}.

Now, since $\Delta(S) \subseteq 2\LQ$, the polynomial $f(x) = (x-\beta)/2$ is in $\Int(S,R)$. Since $R = \HQ$, $\bfh \in R$ and we have
\begin{equation*}
(f\bfh)(\beta) = (\bfh\beta - \beta\bfh)/2 = \big((b_3-b_2)\bfi + (b_1-b_3)\bfj + (b_2-b_1)\bfk\big)/2.
\end{equation*}
As noted above, at least one $b_i$ is even, and at least one $b_j$ is odd. Thus, at least one of $b_3-b_2$, $b_1-b_3$, or $b_2-b_1$ is odd. Therefore, $(f\bfh)(\beta) \notin R$ and $S$ is not a ringset.\\

\noindent\textbf{\underline{Case 2}}: $S \subsetneq \HQ\setminus\LQ$

This is similar to part (1). This time, each $\delta \in \Delta(S)$ has the form $\delta = \tfrac{d_1}{2}\bfi + \tfrac{d_2}{2}\bfj + \tfrac{d_3}{2}\bfk$ with $d_1, d_2, d_3 \in \Z$. The norm of $\delta$ is $N(\delta) = \tfrac{1}{4}(d_1^2+d_2^2+d_3^2)$. Assuming that $4 \mid \Gamma(S)$, we must have $d_1^2+d_2^2+d_3^2 \equiv 0 \pmod{16}$, and hence 4 divides each $d_i$. Consequently, the polynomial $(x-\alpha)/2 \in \Int(S,R)$ for each $\alpha \in S$.

Fix $\beta \in S$. Since $S \subseteq \HQ\setminus\LQ$, we have $\beta = \tfrac{1}{2}(a+b_1\bfi+b_2\bfj+b_3\bfk)$, where $a, b_1, b_2$, and $b_3$ are odd integers. Let $f(x) = (x-\beta)/2 \in \Int(S,R)$. Then, 
\begin{equation*}
(f\bfi)(\beta) = (\bfi \beta - \beta \bfi)/2 = (-2(\tfrac{b_3}{2})\bfj + 2(\tfrac{b_2}{2})\bfk)/2 = (-b_3\bfj + b_2\bfk)/2 \notin R.
\end{equation*}
Thus, $f\bfi \notin \Int(S,R)$, and $S$ is not a ringset.
\end{proof}

To this point, we have focused on powers of $2$ that could divide $\Gamma(S)$. This is because whenever an odd prime $p$ divides $\Gamma(S)$, $S$ will not be a ringset (see Proposition \ref{prop:p divides Gamma} below). It also remains to consider the possibility that $R=\LQ$ and $8 \mid \Gamma(S)$. This case requires a more detailed analysis, and is done in Section \ref{sec:Gamma is 8}.

\begin{Lem}\label{lem:lin poly lemma}
Let $S \subsetneq R$ be reduced. Assume that $\Int(S,R)$ contains a polynomial of the form $f(x)=\gamma(x-\beta)/p$, where $\gamma$, $\beta$, and $p$ satisfy the following conditions:
\begin{enumerate}[(i)]
\item $p$ is an odd prime.
\item $\gamma \in R$ but $\gamma/p \notin R$.
\item $\beta \in S$ satisfies the conclusion of Corollary \ref{cor:prime factors} with respect to $p$.
\end{enumerate}
Then, $S$ is not a ringset.
\end{Lem}
\begin{proof}
Suppose by way of contradiction that $S$ is a ringset. Since $S$ is reduced,  either $S \subseteq \LQ$ or $S \subseteq \HQ\setminus\LQ$ by Remark \ref{Disjoint classes}. Let $\beta = b_0 + b_1\bfi + b_2\bfj + b_3\bfk$, where each $b_i \in \Z$ if $S \subseteq \LQ$, and each $b_i \in \Z + \tfrac{1}{2}$ if $S \subseteq \HQ\setminus\LQ$. By assumption, there is some $1\leq i \leq 3$ such that either $p \nmid b_i$ (if $\beta \in \LQ$) or $p \nmid 2b_i$ (if $\beta \in \HQ\setminus\LQ$).

Since $S$ is a ringset, $\Int(S,R)$ contains $f\bfi$, $f\bfj$, and $f\bfk$. Consider $(f\bfi)(\beta)$. In $\D[x]$, we have $(f\bfi)(x) = (\gamma\bfi x - \gamma \beta \bfi)/p$, and so
\begin{equation*}
(f\bfi)(\beta) = (\gamma\bfi \beta - \gamma \beta \bfi)/p = (\gamma/p)(\bfi\beta-\beta\bfi) = (\gamma/p)(-2b_3\bfj + 2b_2\bfk) \in R.
\end{equation*}
Note that $-2b_3\bfj + 2b_2\bfk \in R$. Multiplying $(f\bfi)(\beta)$ on the right by $\overline{-2b_3\bfj+2b_2\bfk}$ produces $(\gamma/p)(4b_2^2+4b_3^2) \in R$. Performing similar steps with $(f\bfj)(\beta)$ and $(f\bfk)(\beta)$ shows that $(\gamma/p)(4b_1^2+4b_3^2) \in R$ and $(\gamma/p)(4b_1^2+4b_2^2) \in R$. Then, $R$ also contains
\begin{equation*}
(\gamma/p)(8b_1^2) = (\gamma/p)(4b_1^2+4b_2^2) + (\gamma/p)(4b_1^2+4b_3^2) - (\gamma/p)(4b_2^2+4b_3^2).
\end{equation*}
Likewise, $(\gamma/p)(8b_2^2), (\gamma/p)(8b_3^2) \in R$.

By assumption, $\gamma/p \notin R$, so $p$ must divide $8b_i^2$ for each $1 \leq i \leq 3$. Since $p$ is odd, this means that for all $1 \leq i \leq 3$, $p$ divides $b_i$ (if $\beta \in \LQ$), or $2b_i$ (if $\beta \in \HQ\setminus\LQ$). This contradicts Corollary \ref{cor:prime factors}. Therefore, $S$ cannot be a ringset.
\end{proof}

From here, we will prove that whenever an odd prime $p$ divides $\Gamma(S)$, $\Int(S,R)$ contains a polynomial of the form $\gamma(x-\beta)/p$ as in Lemma \ref{lem:lin poly lemma}. To do this, we will use the well-known fact (see e.g.\ \cite[Chap.\ 11]{Voight}) that $R/pR \cong M_2(\F_p)$, the ring of $2 \times 2$ matrices over the finite field $\F_p$.

\begin{Lem}\label{lem:nilpotent matrices}
Let $p$ be an odd prime. Let $A,B \in M_2(\F_p)$ be nonzero and such that $A$, $B$, and $B-A$ are all nilpotent. Then, $A$ and $B$ are nonzero scalar multiples of one another.
\end{Lem}
\begin{proof}
Since $A$ and $B$ are nonzero and nilpotent, they are each similar to $\smat{0}{1}{0}{0}$. So, we may assume without loss of generality that $A = \smat{0}{1}{0}{0}$. Since $B$ is nilpotent, both its trace and determinant are 0. So, $B = \smat{a}{b}{c}{-a}$ for some $a, b, c \in \F_p$ such that $-a^2-bc=0$. We have $0=\det(B-A) = -a^2-bc+c$, so $c=0$. This forces $a=0$. Hence, $B = bA$, and $b \ne 0$ because $B$ is not the zero matrix.
\end{proof}

\begin{Prop}\label{prop:p divides Gamma}
Let $S \subsetneq R$ be reduced. If there is an odd prime $p$ such that $p$ divides $\Gamma(S)$, then $S$ is not a ringset.
\end{Prop}
\begin{proof}
If $S$ is a singleton set, then the result is true by Lemma \ref{lem:basic ringsets}. So, assume that $|S| \geq 2$ and that the odd prime $p$ divides $\Gamma(S)$. Let $\beta \in S$ satisfy the conclusion of Corollary \ref{cor:prime factors} for $p$.

The ring isomorphism $R/pR \cong M_2(\F_p)$ respects minimal polynomials over $\F_p$. Since $S$ is reduced, each element of $\Delta(S)$ has real part 0 and norm divisible by $p$. Hence, the residue of each element of $\Delta(S)$ is nilpotent modulo $p$. Let 
\begin{equation*}
\Delta_\beta = \{\alpha - \beta \mid \alpha \in S, \alpha \ne \beta\} \subseteq \Delta(S).
\end{equation*}
Given $\alpha_1 - \beta, \, \alpha_2 - \beta \in \Delta_\beta$, we have $(\alpha_1-\beta) - (\alpha_2 - \beta) = \alpha_1 - \alpha_2 \in \Delta(S)$. So, all elements in $\Delta_\beta \pmod{p}$ are nilpotent, and have a nilpotent difference. By Lemma \ref{lem:nilpotent matrices}, all the nonzero elements of $\Delta_\beta \pmod{p}$ are $\F_p$-scalar multiples of one another. If $\Delta_\beta \equiv \{0\} \pmod{p}$, then let $\gamma=1$. If not, then there exists $\gamma \in \Delta_\beta$ such that $\gamma \not\equiv 0 \pmod{p}$ and $\gamma(\alpha-\beta) \equiv 0 \pmod{p}$ for all $\alpha \in S$. In either case, the polynomial $\gamma(x-\beta)/p$ is in $\Int(S,R) \setminus R[x]$. By Lemma \ref{lem:lin poly lemma}, we conclude that $S$ is not a ringset.
\end{proof}

We have now proved most cases of the classification theorem for reduced sets that are ringsets.

\begin{Thm}\label{thm:all except 8}
Let $S \subsetneq R$ be reduced.
\begin{enumerate}[(1)]
\item Assume $R = \LQ$. If $S$ is a ringset, then $\Gamma(S) = 2$, $4$, or $8$. If $\Gamma(S) = 2$ or $4$, then $S$ is a ringset.
\item Assume $R = \HQ$. Then, $S$ is a ringset if and only if $\Gamma(S) = 1$ or $2$.
\end{enumerate}
\end{Thm}
\begin{proof}
Apply Theorem \ref{thm:Gamma is 1, 2, or 4} and Propositions \ref{prop:powers of 2 divide Gamma} and \ref{prop:p divides Gamma}. Note that when $S \subseteq \LQ$, each element of $\Delta(S)$ has even norm (see Remark \ref{rem:Gamma}), so $\Gamma(S) \ne 1$ in this case.
\end{proof}

\section{\texorpdfstring{Reduced Sets $S$ with $\Gamma(S)=8$}{Reduced Sets S with G(S) =8}}\label{sec:Gamma is 8}

Throughout this section, $R = \LQ$. Here, we deal with the exceptional case in which a reduced set $S \subseteq \LQ$ has $\Gamma(S)=8$. In this situation, $S$ may or may not be a ringset.

\begin{Ex}\label{ex:Gamma=8}
Recall that $R=\LQ$, so that $R^\times = \{\pm 1, \, \pm\bfi, \, \pm\bfj, \, \pm\bfk\}$.
\begin{enumerate}[(1)]
\item Let $\alpha = \bfi+\bfj+\bfk$ and let
\begin{equation*}
S = \{u \alpha u^{-1} \mid u \in R^\times\} = \{\bfi+\bfj+\bfk, \; \bfi-\bfj-\bfk, \; -\bfi+\bfj-\bfk, \; -\bfi-\bfj+\bfk\}.
\end{equation*}
Then, $S$ is reduced, $\Gamma(S)=8$, and $S$ is a ringset because it is a full conjugacy class in $R$ (Lemma \ref{lem:basic ringsets}).

\item Let $\alpha_1 = 2\bfi+3\bfj+4\bfk$, $\alpha_2 = -2\bfi+3\bfj+4\bfk$, and $\beta = -5\bfj-2\bfk$. Take $T = \{\alpha_1, \alpha_2, \beta\}$. Then, $T$ is reduced and $\Gamma(T)=8$. Modulo 4, we have
\begin{equation*}
\alpha_1 - \beta \equiv 2\bfi + 2\bfk \equiv \alpha_2 - \beta \pmod{4}.
\end{equation*}
Since $(1+\bfj)(2\bfi+2\bfk) \equiv 0 \pmod{4}$, this shows that $f(x)=(1+\bfj)(x-\beta)/4 \in \Int(T,R)$. However, $(f\bfi)(\beta) = (-4-10\bfi+4\bfj-10\bfk)/4 \notin R$. Thus, $T$ is not a ringset.
\end{enumerate}
\end{Ex}

The key difference between $S$ and $T$ in Example \ref{ex:Gamma=8} appears when we consider $\Delta(S)$ and $\Delta(T)$ modulo 4. We have
\begin{align*}
\Delta(S) &\equiv \{0, \, 2\bfi+2\bfj, \, 2\bfi+2\bfk, \, 2\bfj+2\bfk\} \pmod{4}, \text{ and}\\
\Delta(T) &\equiv \{0, \, 2\bfi+2\bfk\} \pmod{4}.
\end{align*}
Essentially, the fact that $\Delta(T)$ touches only one nonzero residue class modulo 4 makes it possible to construct a polynomial like $f \in \Int(T,R)$ above, and yet have one of $f\bfi$, $f\bfj$, or $f\bfk$ fail to be integer-valued on $T$. In contrast, when $\Delta(S)$ touches more than one nonzero residue class modulo 4, this sort of obstruction does not appear, and we can prove that $S$ is a ringset. The remainder of this section builds up the theory necessary to prove these statements.

\begin{Lem}\label{lem:one class no ringset}
Let $S \subsetneq \LQ$ be reduced. Assume that $\Delta(S) \pmod{4}$ is contained in one of $\{0, \; 2\bfi+2\bfj\}$, $\{0, \; 2\bfi + 2\bfk\}$, or $\{0, \; 2\bfj+2\bfk\}$. Then, $S$ is not a ringset.
\end{Lem}
\begin{proof}
Assume without loss of generality that $\Delta(S) \pmod{4} \subseteq \{0, \; 2\bfi+2\bfj\}$. Then, $(1+\bfk)(\alpha-\beta) \equiv 0 \pmod{4}$ for all $\alpha, \beta \in S$. So, the polynomial $(1+\bfk)(x-\alpha)/4 \in \Int(S,\LQ)$ for all $\alpha \in S$. Since $S$ is reduced, by Corollary \ref{cor:prime factors} there exists $\beta = a+b_1\bfi+b_2\bfj+b_3\bfk \in S$ such that at least one of $b_1$, $b_2$, or $b_3$ is odd. Let $f(x) = (1+\bfk)(x-\beta)/4 \in \Int(S,\LQ)$. Evaluating $(f\bfi)(\beta)$ and $(f\bfj)(\beta)$, we find that
\begin{align*}
(f\bfi)(\beta) &= (1+\bfk)(\bfi \beta - \beta \bfi)/2 = (-2b_2 + 2b_3\bfi - 2b_3\bfj + 2b_2\bfk)/4, \text{ and}\\
(f\bfj)(\beta) &= (1+\bfk)(\bfj \beta - \beta \bfj)/2 = (2b_1 + 2b_3\bfi + 2b_3\bfj - 2b_1\bfk)/4.
\end{align*}
Since at least one of $b_1$, $b_2$, or $b_3$ must be odd, either $f\bfi$ or $f\bfj$ is not in $\Int(S,\LQ)$. Thus, $S$ is not a ringset. 

If, modulo 4, $\Delta(S) \subseteq \{0, \; 2\bfi + 2\bfk\}$ (respectively, $\{0, \; 2\bfj+2\bfk\}$), then we reach the same conclusion by considering $f(x) = (1+\bfj)(x-\beta)/4$ (respectively, $f(x) = (1+\bfi)(x-\beta)/4$). In all cases, $f \in \Int(S,\LQ)$, but one of $f\bfi, f\bfj$, or $f\bfk$ will fail to be in $\Int(S,\LQ)$.
\end{proof}

Lemma \ref{lem:one class no ringset} will be our main tool for showing that some sets $S$ with $\Gamma(S)=8$ are not ringsets. Our next step is to examine the imaginary coefficients of the elements of such a set $S$.

\begin{Lem}\label{lem:Gamma is 8 lem1}
Let $S \subsetneq \LQ$ be reduced and such that $\Gamma(S)=8$. Let $\alpha=a + a_1\bfi + a_2\bfj + a_3\bfk$ and $\beta = a+b_1\bfj+b_2\bfj+b_3\bfk$ be elements of $S$.
\begin{enumerate}[(1)]
\item For all $1 \leq i \leq 3$, $a_i$ and $b_i$ have the same parity.
\item For at least one $i \in \{1,2,3\}$, both $a_i$ and $b_i$ are odd.
\item Either $a_i - b_i \equiv 0 \pmod{4}$ for all $1 \leq i \leq 3$, or $a_i - b_i \equiv 0 \pmod{4}$ for exactly one $i \in \{1, 2, 3\}$.
\end{enumerate}
\end{Lem}
\begin{proof}
Since $\Gamma(S)=8$, we have
\begin{equation}\label{eq:N mod 8}
N(\alpha-\beta) \equiv (a_1-b_1)^2 + (a_2-b_2)^2 + (a_3-b_3)^2 \equiv 0 \pmod{8}.
\end{equation}
First, \eqref{eq:N mod 8} implies that $a_i-b_i$ is even for each $i$. Consequently, some $a_i$ must be odd; if not, then $b_i$ is even for every $1 \leq i \leq 3$ and for every $\beta \in S$, which violates Corollary \ref{cor:prime factors}. This proves (1) and (2).

Next, if $a_i-b_i \equiv 0 \pmod{4}$ for exactly zero or two values of $i \in \{1,2,3\}$, then $N(\alpha-\beta) \equiv 4 \pmod{8}$. This contradicts \eqref{eq:N mod 8}, so (3) holds. 
\end{proof}

\begin{Lem}\label{lem:not a ringset 1}
Let $S \subsetneq \LQ$ be reduced and such that $\Gamma(S)=8$. If there exists $\beta = a+b_1\bfj+b_2\bfj+b_3\bfk \in S$ such that exactly one of $b_1$, $b_2$, or $b_3$ is even, then $S$ is not a ringset. 
\end{Lem}
\begin{proof}
Assume that such a $\beta \in S$ exists. Without loss of generality, assume that $b_1$ is even while $b_2$ and $b_3$ are odd. Let $\alpha=a + a_1\bfi + a_2\bfj + a_3\bfk \in S$.
Note that $a_1$ must be even and each of $a_2$ and $a_3$ must be odd, by Lemma \ref{lem:Gamma is 8 lem1}(1).

We claim that $a_1 \equiv b_1 \pmod{4}$. To see this, consider the norms of $\alpha$ and $\beta$ mod 8. Since $\alpha, \beta \in S$, we have $N(\alpha) = N(\beta)$. Moreover, each of $a_2, a_3, b_2$, and $b_3$ is odd, so $a_2^2 + a_3^2 \equiv 2 \equiv b_2^2+b_3^2 \pmod{8}$. It follows that $a_1^2 \equiv b_1^2 \pmod{8}$. Because both $a_1$ and $b_1$ are even, they must be equivalent modulo 4. Thus,
\begin{equation*}
\alpha - \beta \equiv (a_1-b_1)\bfi + (a_2-b_2)\bfj + (a_3-b_3)\bfk \equiv (a_2-b_2)\bfj + (a_3-b_3)\bfk \pmod{4}.
\end{equation*}
By Lemma \ref{lem:Gamma is 8 lem1}(3), we see that $\alpha - \beta \pmod{4} \in \{0, \; 2\bfj+2\bfk\}$. This holds for all $\alpha \in S$. If $\alpha_1, \alpha_2 \in S$, then $\alpha_1-\alpha_2 = (\alpha_1-\beta) + (\alpha_2-\beta)$. Thus, $\Delta(S) \pmod{4}$ is contained in $\{0, \; 2\bfj+2\bfk\}$, and hence $S$ is not a ringset by Lemma \ref{lem:one class no ringset}.

If $b_2$ (respectively, $b_3$) is even, then $\Delta(S) \pmod{4}$ is contained in $\{0, \; 2\bfi+2\bfk\}$ (respectively, $\{0, \; 2\bfi+2\bfj\}$). So, we reach the same conclusion in those cases.
\end{proof}

\begin{Lem}\label{lem:not a ringset 2}
Let $S \subsetneq \LQ$ be reduced and such that $\Gamma(S)=8$. If there exists $\beta = a+b_1\bfj+b_2\bfj+b_3\bfk \in S$ such that exactly two of $b_1$, $b_2$, or $b_3$ are even, then $S$ is not a ringset. 
\end{Lem}
\begin{proof}
This is similar to the proof of the previous lemma. Assume that such a $\beta \in S$ exists. Without loss of generality, assume that $b_3$ is odd, while $b_1$ and $b_2$ are even. Let $\alpha=a + a_1\bfi + a_2\bfj + a_3\bfk \in S$. By Lemma \ref{lem:Gamma is 8 lem1}(3), the residue of $\alpha-\beta$ modulo 4 must lie in $\{0, \; 2\bfi+2\bfj, \; 2\bfi+2\bfk, \; 2\bfj+2\bfk\}$. We will show that $\alpha - \beta \pmod{4} \in \{0, \; 2\bfi+2\bfj\}$.

Suppose that $\alpha-\beta \equiv 2\bfi+2\bfk \pmod{4}$. So, $a_2 \equiv b_2 \pmod{4}$, and hence $a_2^2 \equiv b_2^2 \pmod{8}$. We have $a_3^2 \equiv b_3^2 \pmod{8}$ because both $a_3$ and $b_3$ are odd. Since $N(\alpha) = N(\beta)$, these equivalences force $a_1^2 \equiv b_1^2 \pmod{8}$. However, this is impossible. Both $a_1$ and $b_1$ are even, and $a_1-b_1 \equiv 2 \pmod{4}$, so $a_1^2 \not\equiv b_1^2 \pmod{8}$. We will reach a similar contradiction if $\alpha-\beta \equiv 2\bfj+2\bfk \pmod{4}$ by examining $a_2$ and $b_2$.

Our work so far shows that $\alpha - \beta \pmod{4} \in \{0, \; 2\bfi+2\bfj\}$. As in Lemma \ref{lem:not a ringset 1}, this implies that $\Delta(S) \pmod{4} \subseteq \{0, \; 2\bfi+2\bfj\}$. Hence, $S$ is not a ringset by Lemma \ref{lem:one class no ringset}. We obtain the same result if $b_1$ (respectively, $b_2$) is odd, because $\Delta(S) \pmod{4} \subseteq \{0, \; 2\bfj+2\bfk\}$ (respectively, $\{0, \; 2\bfi+2\bfk\}$).
\end{proof}

From Lemmas \ref{lem:not a ringset 1} and \ref{lem:not a ringset 2}, we see that if $S$ is reduced, $\Gamma(S) = 8$, and $S$ is a ringset, then it is necessary that every imaginary coefficient of every element of $S$ is odd. Unfortunately, this condition is not sufficient. 

\begin{Ex}\label{ex:all odds}
Let $\alpha = \bfi+\bfj+\bfk$, $\beta = \bfi-\bfj-\bfk$, and $S=\{\alpha, \beta\}$. Then, $S$ is reduced, $\Gamma(S) = 8$ and $\Delta(S) = \{0, \pm(2\bfj+2\bfk)\}$. By Lemma \ref{lem:one class no ringset}, $S$ is not a ringset with respect to $\LQ$. We note, however, that $S$ is a ringset with respect to $\HQ$. In this larger ring, $S$ is not reduced, because $S+1 \subseteq 2\HQ$. Taking $T = \{(1+\bfi+\bfj+\bfk)/2, \, (1+\bfi-\bfj-\bfk)/2\}$, we have $S = -1+2T$. In $\HQ$, $T$ is reduced and $\Gamma(T) = 2$. By Theorem \ref{thm:all except 8} and Lemma \ref{lem:translation}, both $T$ and $S$ are ringsets with respect to $\HQ$.
\end{Ex}

\begin{Thm}\label{thm:Gamma=8 full thm}
Let $S \subsetneq \LQ$ be reduced and such that $\Gamma(S)=8$. Then, $S$ is a ringset if and only if there exist $\delta_1, \delta_2 \in \Delta(S)$ such that $\delta_1$ and $\delta_2$ are congruent modulo 4 to distinct residues in $\{2\bfi+2\bfj, \; 2\bfi+2\bfk, \; 2\bfj+2\bfk\}$.
\end{Thm}
\begin{proof}
Note that by Lemma \ref{lem:Gamma is 8 lem1}, $\Delta(S) \pmod{4}$ is contained in $\{0, \,2\bfi+2\bfj, \, 2\bfi+2\bfk, \, 2\bfj+2\bfk\}$.

$(\Rightarrow)$ This follows from Lemma \ref{lem:one class no ringset}. 

$(\Leftarrow)$ Assume that the desired $\delta_1$ and $\delta_2$ exist. Without loss of generality, assume that $\delta_1 \equiv 2\bfi+2\bfj$ and $\delta_2 \equiv 2\bfi+2\bfk$. From Lemmas \ref{lem:not a ringset 1} and \ref{lem:not a ringset 2}, we know that every imaginary coefficient of each element of $S$ must be odd.

Let $r(x)=\gamma_1x+\gamma_0 \in \Int(S,\LQ)$. We will show that $r \in \Int(S^\ast,\LQ)$, where 
\begin{equation*}
S^\ast=\{u \alpha u^{-1} \mid \alpha \in S, u \in \LQ^\times\} = \{\alpha, \, -\bfi\alpha\bfi, \, -\bfj\alpha\bfj, \, -\bfk\alpha\bfk\}.
\end{equation*} 
As in the proof of Theorem \ref{thm:Gamma is 1, 2, or 4}(3), $\overline{\alpha-\beta} \in 2\LQ$ for all $\alpha, \beta \in S$.  In analogy with \eqref{eq:Gamma=4}, there exist $t \geq 1$, $n_i \in \Z$ and $\alpha_i, \beta_i \in S$ for $1 \leq i \leq t$, such that
\begin{equation*}
4\gamma_1 = \tfrac{1}{2}(8\gamma_1) = \sum_{i=1}^t n_i\gamma_1(\alpha_i - \beta_i)\Big(\dfrac{\;\overline{\alpha_i-\beta_i}\;}{2}\Big).
\end{equation*}
Thus, $4\gamma_1 \in \LQ$ because $\gamma_1(\alpha_i-\beta_i)$, $\big(\,\overline{\alpha_i-\beta_i}\,\big)/2 \in \LQ$ for each $i$.

Now, for all $\alpha, \beta \in S$, $\gamma_1(\alpha-\beta) = r(\alpha) - r(\beta) \in \LQ$. Thus, $\gamma_1 \delta \in \LQ$ for all $\delta \in \Delta(S)$. Let $\varepsilon_1 \in \LQ$ be such that $\delta_1 = 2\bfi+2\bfj + 4\varepsilon_1$. Since $\gamma_1\delta_1, 4\gamma_1 \in \LQ$, we see that $\gamma_1(2\bfi+2\bfj) \in \LQ$. Similarly, $\gamma_1(2\bfi+2\bfk) \in \LQ$ because $\gamma_1\delta_2 \in \LQ$.

Let $\alpha=a+a_1\bfi+a_2\bfj+a_3\bfk \in S$. Recall that $a_i$ is odd for all $1 \leq i \leq 3$. This means that $2a_i \equiv 2 \pmod{4}$ for each $i$. Let $\alpha_1, \alpha_2 \in \LQ$ be such that
\begin{equation*}
2a_1\bfi+2a_2\bfj = 2\bfi+2\bfj+4\alpha_1 \quad \text{ and } \quad 2a_1\bfi+2a_3\bfk = 2\bfi+2\bfk+4\alpha_2.
\end{equation*}
We have
\begin{align*}
r(\alpha) - r(-\bfk\alpha\bfk) &= \gamma_1(2a_1\bfi+2a_2\bfj) = \gamma_1(2\bfi+2\bfj) + 4\gamma_1\alpha_1 \in \LQ, \text{ and}\\
r(\alpha) - r(-\bfj\alpha\bfj) &= \gamma_1(2a_1\bfi+2a_3\bfk) = \gamma_1(2\bfi+2\bfk) + 4\gamma_1\alpha_2 \in \LQ.
\end{align*}
These equations imply that $r(-\bfk\alpha\bfk), r(-\bfj\alpha\bfj) \in \LQ$. Finally,
\begin{equation*}
r(\alpha) - r(-\bfi\alpha\bfi) = \gamma_1(2a_2\bfj+2a_3\bfk) = \gamma_1(2a_1\bfi+2a_2\bfj) - \gamma_1(2a_1\bfi+2a_3\bfk) + 4\gamma_1a_3\bfk,
\end{equation*}
so $r(-\bfi\alpha\bfi) \in \LQ$ as well. This shows that $r \in \Int(S^\ast,R)$, so $S$ is a ringset by Lemma \ref{lem:S star and linear polys}.
\end{proof}

\section{Closing Remarks and Infinite Subsets}\label{sec:last}

We now have all the tools necessary to determine whether or not a finite subset of $R$ is a ringset.

\begin{Alg}\label{alg:ringset}
Let $S \subseteq R$ be finite and nonempty. To decide if $S$ is a ringset or not, follow these steps.
\begin{enumerate}
\item Partition $S$ as $S = \bigcup_{i=1}^t S_i$, where each $S_i$ is a nonempty subset of a minimal polynomial class $\mcC_R(m_i)$, and each $m_i$ is distinct. By Corollary \ref{cor:min poly decomp}, $S$ is a ringset if and only if each $S_i$ is a ringset.

\item Any $S_i \subseteq \Z$ is a ringset by Lemma \ref{lem:basic ringsets} and can be ignored. If any $S_i \subseteq R \setminus \Z$ is a singleton set, then $S_i$---and hence $S$---is not a ringset.

\item For each $S_i \subseteq R \setminus \Z$ that is not a singleton set, find an associated reduced set $T_i$. Compute $\Gamma(T_i)$ and apply Theorem \ref{thm:main} to determine whether $T_i$ is a ringset. If each $T_i$ is a ringset, then $S$ is a ringset; and if some $T_i$ is not a ringset, then $S$ is not a ringset.
\end{enumerate}
\end{Alg}

The strategy used in Algorithm \ref{alg:ringset} depends heavily on the assumption that $S$ is finite. Recall that by Remark \ref{C_R(m) is finite}, each minimal polynomial class $\mcC_R(m)$ in $R$ is finite. Moreover, there is no guarantee that Theorem \ref{thm:separation} and Corollary \ref{cor:min poly decomp} will hold for an infinite set. As we now demonstrate, there exist sets $S$ with a minimal polynomial partition $S = \bigcup_{i=1}^\infty S_i$ such that each $S_i$ is not a ringset, but $S$ is a ringset. Rather than deal directly with elements of $S$ in $R$, we will focus on calculations in the finite residue rings $R/nR$, where $n \in \Z$, $n \geq 2$.

\begin{Def}\label{def:null ideal}
Let $A$ be a ring and $S \subseteq A$. The \textit{null ideal} of $S$ in $A$ is $\mcN(S,A):=\{f \in A[x] \mid f(\alpha)=0 \text{ for all } \alpha \in S\}$.
\end{Def}

It is straightforward to verify that $\mcN(S,A)$ is a left ideal of $A[x]$. Whether it is a two-sided ideal depends on $S$ and $A$; this is a question that recently has been pursued for rings of matrices over finite fields \cite{RissnerWerner, SwartzWerner, WernerNull}. Our interest in null ideals lies in their connection to integer-valued polynomials.

\begin{Lem}\label{lem:null ideals IVP}
Let $f \in \D[x]$ and write $f(x)$ as $f(x) = g(x)/n$, where $g \in R[x]$ and $n \in \Z$, $n \geq 1$. Let $S \subseteq R$. Use a tilde to denote passage from $R$ to $R/nR$. Then, $f \in \Int(S,R)$ if and only if $\widetilde{g} \in \mcN(\widetilde{S},R/nR)$.
\end{Lem}
\begin{proof}
For each $\alpha \in S$, we have $f(\alpha) \in R$ if and only if $g(\alpha) \in nR$, if and only if $\widetilde{g}(\widetilde{\alpha}) = 0$ in $R/nR$.
\end{proof}

\begin{Lem}\label{lem:mult by units}
Let $S \subseteq R$.
\begin{enumerate}[(1)]
\item Assume $R = \LQ$. Then, $S$ is a ringset if and only if for all $f \in \Int(S,R)$, each of $f\bfi$, $f\bfj$, and $f\bfk$ is in $\Int(S,R)$.
\item Assume $R = \HQ$. Then, $S$ is a ringset if and only if for all $f \in \Int(S,R)$, each of $f\bfi$, $f\bfj$, $f\bfk$, and $f\bfh$ is in $\Int(S,R)$.
\end{enumerate}
\end{Lem}
\begin{proof}
(1) By \cite[Prop.\ 40]{WernerSurvey}, $S$ is a ringset if and only if $f\alpha \in \Int(S,R)$ for all $f \in \Int(S,R)$ and all $\alpha \in R$. Given $\alpha \in \LQ$, we have $\alpha = a_0 + a_1\bfi + a_2\bfj + a_3\bfk$ for some $a_i \in \Z$, $0\leq i \leq 3$. Then,
\begin{equation}\label{eq:right mult}
f\alpha = a_0f + a_1f\bfi + a_2f\bfj + a_3f\bfk.
\end{equation}
Clearly, if $S$ is a ringset, then $f\bfi$, $f\bfj$, and $f\bfk$ are all in $\Int(S,\LQ)$. Conversely, if $f\bfi, f\bfj, f\bfk \in \Int(S,\LQ)$, then by \eqref{eq:right mult}, $f\alpha \in \Int(S,\LQ)$ for any choice of integers $a_i$, $0 \leq i \leq 3$. Thus, $S$ is a ringset.

(2) This is identical to the proof of (1) after noting that $\alpha \in \HQ$ can be written as
\begin{equation*}
\alpha = a_0 + a_1\bfi + a_2\bfj + a_3\bfk + e\bfh,
\end{equation*}
where $a_i \in \Z$ for $0 \leq i \leq 3$, and $e \in \{0,1\}$.
\end{proof}

\begin{Prop}\label{prop:infinite ringset}
Let $S \subseteq R$ be such that for all $n \in \Z$, $n \geq 2$, and all $\alpha \in S$, there exists $\beta \in S$ such that 
\begin{enumerate}[(i)]
\item the minimal polynomials of $\alpha$ and $\beta$ are equivalent modulo $n$, and
\item $\alpha - \beta \equiv \pm(\bfi-\bfj) \pmod{n}$.
\end{enumerate}
Then, $S$ is a ringset.
\end{Prop}
\begin{proof}
Let $f \in \Int(S,R)$ and write $f(x) = g(x)/n$, where $g \in R[x]$ and $n \in \Z$, $n \geq 1$. From Lemmas \ref{lem:null ideals IVP} and \ref{lem:mult by units}, it suffices to show that the images of $g\bfi$, $g\bfj$, $g\bfk$, and $g\bfh$ (if $R = \HQ$) are in the null ideal of the image of $S$ in $R/nR$. This is clear if $n=1$, so assume that $n \geq 2$.

Fix $\alpha \in S$ with minimal polynomial $m$. Let $q, r \in R[x]$ be such that $g=qm+r$ and $r = \gamma_1x + \gamma_0$ for some $\gamma_1, \gamma_0 \in R$. Then, $r(\alpha) = g(\alpha)$, and $g(\alpha) \equiv 0 \pmod{n}$. By assumption, there exists $\beta \in S$ such that $m(\beta) \equiv m(\alpha) \equiv 0 \pmod{n}$ and $\alpha-\beta \equiv \pm(\bfi-\bfj) \pmod{n}$. Since $\beta \in S$, $g(\beta) \equiv 0 \pmod{n}$. Hence, $r(\beta) \equiv 0 \pmod{n}$, and so
\begin{equation}\label{eq:i-j mod n}
0 \equiv r(\alpha) - r(\beta) \equiv \gamma_1(\alpha-\beta) \equiv \pm\gamma_1(\bfi-\bfj) \pmod{n}.
\end{equation}
Since $(\bfi-\bfj)(-\bfi+\bfj)=2$, this implies that $2\gamma_1 \equiv 0 \pmod{n}$. Furthermore, if $R=\HQ$, then $1+\bfi = (\bfi-\bfj)(-\bfj\bfh\bfj)$ and \eqref{eq:i-j mod n} implies that
\begin{equation}\label{eq:j-k k-i mod n}
\gamma_1(1+\bfi) \equiv 0 \pmod{n}.
\end{equation}

Let $u \in \{\bfi, \bfj, \bfk, \bfh\}$. We wish to show that $(gu)(\alpha) \equiv 0 \pmod{n}$. Note that $(gu)(\alpha) = (ru)(\alpha)$, and $r(\alpha)u \equiv 0 \pmod{n}$. So,
\begin{equation*}
(ru)(\alpha) \equiv (ru)(\alpha) - r(\alpha)u \equiv \gamma_1(u\alpha-\alpha u) \pmod{n}.
\end{equation*}
If $R = \LQ$, then $u \in \{\bfi, \bfj, \bfk\}$ and $u\alpha - \alpha u \in 2R$ by Lemma \ref{lem:difference of conjugates}. From \eqref{eq:i-j mod n}, we see that $(ru)(\alpha) \equiv 0 \pmod{n}$. If $R = \HQ$, then $u\alpha - \alpha u \in (1+\bfi)R$, so $(ru)(\alpha) \equiv 0 \pmod{n}$ by \eqref{eq:j-k k-i mod n}. Thus, $(gu)(\alpha) \equiv 0 \pmod{n}$ for each choice of $u$.

The arguments in the previous paragraphs hold for any $\alpha \in S$ and all $u \in \{\bfi, \bfj, \bfk, \bfh\}$. Thus, $f\bfi, f\bfj, f\bfk$ (and $f\bfh$, if necessary) are all in $\Int(S,R)$, and $S$ is a ringset by Lemma \ref{lem:mult by units}.
\end{proof}

We close the paper with the promised example of an infinite ringset such that each set in its minimal polynomial partition is not a ringset.

\begin{Ex}\label{ex:infinite ringset}
For each integer $n \geq 2$, let
\begin{equation*}
T_n := \{a+\bfi \mid n^2-n \leq a \leq n^2-1\} \cup \{a+\bfj \mid n^2\leq a \leq n^2+n-1\}.
\end{equation*}
Let $S = \bigcup_{n=2}^\infty T_n$. Observe that each element of $S$ has the form $a+\bfi$ or $a + \bfj$ for some $a \in \Z$, $a \geq 2$. Furthermore, for each such $a$, there is a unique element of $S$ having real part $a$. Thus, each element of $S$ lies in a distinct minimal polynomial class of $R$. So, $S = \bigcup_{\alpha \in S} \{\alpha\}$ is the minimal polynomial partition of $S$, and each $\{\alpha\}$ is not a ringset by Lemma \ref{lem:basic ringsets}.

Consider $T_n$ reduced modulo $n$. For each $n \geq 2$,
\begin{equation*}
T_n \equiv \{a+\bfi, \, a+\bfj \mid 0 \leq a \leq n-1\} \pmod{n}.
\end{equation*}
Let $\alpha \in S$. Given $n$, there exists $0 \leq a_n \leq n-1$ such that either
\begin{equation*}
\alpha \equiv a_n+\bfi \pmod{n} \quad \text{ or } \quad \alpha \equiv a_n+\bfj \pmod{n}.
\end{equation*}
In either case, there exists $\beta \in T_n$ such that $\alpha - \beta \equiv \pm(\bfi-\bfj) \pmod{n}$ and the minimal polynomials of $\alpha$ and $\beta$ are equivalent modulo $n$. By Proposition \ref{prop:infinite ringset}, $S$ is a ringset.
\end{Ex}

\bibliographystyle{plainurl}
\bibliography{Index}
 
\end{document}